\documentclass{amsart}

\usepackage{amssymb}

\usepackage{graphicx,bbm,tikz}

\usepackage{enumitem}

\newtheorem{theorem}{Theorem}[section]
\newtheorem{lemma}[theorem]{Lemma}
\newtheorem{corollary}[theorem]{Corollary}

\theoremstyle{definition}
\newtheorem{definition}[theorem]{Definition}
\newtheorem{proposition}[theorem]{Proposition}

\theoremstyle{remark}

\numberwithin{equation}{section}

\newcommand{\Var}{\mathbf{Var}} 
\newcommand{\E}{\mathbf{E}}     
\newcommand{\Prob}{\mathbf{P}}  

\newcommand{\bN}{\mathbb{N}}    
\newcommand{\bR}{\mathbb{R}}    
\newcommand{\bZ}{\mathbb{Z}}    

\newcommand{\cA}{\mathcal{A}}   
\newcommand{\cC}{\mathcal{C}}   
\newcommand{\cO}{\mathcal{O}}   

\newcommand{\mK}{\mathsf{K}}    
\newcommand{\mL}{\mathsf{L}}    
\newcommand{\mF}{\mathsf{F}}    
\newcommand{\mG}{\mathsf{G}}    

\newcommand{\llbracket}{[\![}   
\newcommand{\rrbracket}{]\!]}   

\begin{document}

\title[Central limit theorem for descents and major indices]{A central limit theorem for descents and major indices in fixed conjugacy classes of $S_n$}

\author{Gene B. Kim}
\address{Department of Mathematics, University of Southern California, Los Angeles, CA 90089}
\email{genebkim@usc.edu}

\author{Sangchul Lee}
\address{Department of Mathematics, University of California, Los Angeles, Los Angeles, CA 90095}
\email{sos440@math.ucla.edu}


\date{November 11, 2018}


\begin{abstract}
    The distribution of descents in fixed conjugacy classes of $S_n$ has been studied, and it is shown that its moments have interesting properties. Kim and Lee showed, by using Curtiss' theorem and moment generating functions, how to prove a central limit theorem for descents in arbitrary conjugacy classes of $S_n$. In this paper, we prove a modified version of Curtiss' theorem to shift the interval of convergence in a more convenient fashion and use this to show that the joint distribution of descents and major indices is asymptotically bivariate normal.
\end{abstract}

\maketitle


\section{Introduction}

The theory of descents in permutations has been studied thoroughly and is related to many questions. In \cite{Knuth}, Knuth connected descents with the theory of sorting and the theory of runs in permutations, and in \cite{Diaconis1}, Diaconis, McGrath, and Pitman studied a model of card shuffling in which descents play a central role. Bayer and Diaconis also used descents and rising sequences to give a simple expression for the chance of any arrangement after any number of shuffles and used this to give sharp bounds on the approach to randomness in \cite{Bayer}. Garsia and Gessel found a generating function for the joint distribution of descents, major index, and inversions in \cite{Garsia}, and Gessel and Reutenauer showed that the number of permutations with given cycle structure and descent set is equal to the scalar product of two special characters of the symmetric group in \cite{Gessel}. Diaconis and Graham also explained Peirce's dyslexic principle using descents in~\cite{DiaconisGraham}. Petersen also has an excellent and very thorough book on Eulerian numbers \cite{Petersen}.

\begin{definition}
    A permutation $\pi \in S_n$ has a \textit{descent} at position $i$ if $\pi(i) > \pi(i+1)$, where $i = 1, \dots, n-1$. The \textit{descent number} of $\pi$, denoted $d(\pi)$, is defined as the number of all descents of $\pi$ plus 1. The \textit{major index} of $\pi$, denoted $maj(\pi)$, is the sum of the positions at which $\pi$ has a descent.
\end{definition}

It is well known (\cite{Diaconis2}) that the distribution of $d(\pi)$ in $S_n$ is asymptotically normal with mean $\frac{n+1}{2}$ and variance $\frac{n+1}{12}$. Fulman also used Stein's method to show that the number of descents of a random permutation satisfies a central limit theorem with error rate $n^{-1/2}$ in \cite{Fulman2}. In \cite{Vatutin}, Vatutin proved a central limit theorem for $d(\pi) + d(\pi^{-1})$, where $\pi$ is a random permutation.

Fulman~\cite{Fulman1} proved that the distribution of descents in conjugacy classes with large cycles is asymptotically normal, and Kim~\cite{Kim} proved that descents in fixed point free involutions (matchings) is also asymptotically normal. After the latter result was proved, Diaconis~\cite{DiaconisFulman} conjectured that there are asymptotic normality results for descents in conjugacy classes that are fixed point free. Kim and Lee, in~\cite{KimLee}, proved a more general result that descents in arbitrary conjugacy classes are asymptotically normal, where the parameters depend only on the ratio of fixed points to $n$.

There have been central limit theorems proven about major indices as well. In~\cite{Chen}, Chen and Wang also use generating functions and Curtiss' theorem to prove asymptotic normality results about major indices of derangements. Billey et al, in~\cite{Billey}, consider the distribution of major indices on standard tableaux of arbitrary straight shape and certain skew shapes.

In this paper, we show that the joint distribution of descents and major indices in any conjugacy classes of $S_n$ is asymptotically bivariate normal. The precise formulation of the main statement is as follows.

\begin{theorem}
    \label{thm:main_01}
    For each conjugacy class $\cC_{\lambda}$ of $S_n$, write $\alpha_1$ for the density of fixed points of any permutations in $\cC_{\lambda}$ and define
    \begin{align*}
        W_{\lambda} = \left( \frac{d(\pi) - \frac{1-\alpha_1^2}{2}n}{n^{1/2}}, \frac{maj(\pi) - \frac{1-\alpha_1^2}{4}n^2}{n^{3/2}} \right),
    \end{align*}
    where $\pi$ is chosen uniformly at random from $\cC_{\lambda}$. Then, along any sequence of~$\cC_{\lambda}$'s such that $n \to \infty$ and $\alpha_1 \to \alpha \in [0, 1]$, $W_{\lambda}$ converges in distribution to a bivariabe normal distribution of zero mean and the covariance matrix $\Sigma_{\alpha}$ depending only on~$\alpha$.
\end{theorem}

We will need two major ingredients for this; one is a modification of Curtiss' theorem relating pointwise convergence of moment generating function (m.g.f.) to the convergence in distribution of corresponding random variables, and the other is a uniform estimate on the m.g.f.\ of the joint distribution of descents and major index. The asymptotic joint normality will then follow as an immediate corollary. This uniform estimate will be strong enough to prove an analogous result for a more general class of subsets of of $S_n$, encompassing the asymptotical joint normality for derangements.

\begin{theorem}
    \label{thm:main_02}
    Suppose that $A_n$ is a subset of $S_n$ which is invariant under conjugation and that all $\pi \in A_n$ have the same number of fixed points. Denote by $\alpha_{1,n}$ the common density of fixed points of elements in $A_n$, and define
    \begin{align*}
        W_n = \left( \frac{d(\pi) - \frac{1-\alpha_{1,n}^2}{2}n}{n^{1/2}}, \frac{maj(\pi) - \frac{1-\alpha_{1,n}^2}{4}n^2}{n^{3/2}} \right),
    \end{align*}
    where $\pi$ is chosen uniformly at random from $A_n$. If $\alpha_{1,n} \to \alpha \in [0, 1]$ as $n\to\infty$, then $W_n$ converges in distribution to a bivariate normal distribution of zero mean and the covariance matrix $\Sigma_{\alpha}$.
\end{theorem}

This paper is organized as follows: In Section 2, we modify Curtiss' theorem in the form which is applicable to our proof. In Section 3, we establish a formula for the joint generating function of $(d(\pi), maj(\pi))$ for $\pi$ in the conjugacy class $\cC_{\lambda}$. In Section 4, we analyze this formula analytically to provide a uniform estimate on the m.g.f.\ $M_{W_{\lambda}}$ and then apply the modified Curtiss' theorem to conclude both main theorems.


\section{A modification of Curtiss' theorem}

For a random variable $ X $ taking values in $ \bR^{d} $, its moment generating function (m.g.f.) is defined as
\begin{align*}
	M_{X}(s) = \E\left[ e^{\langle s, X \rangle} \right], \qquad s \in \bR^d.
\end{align*}
In his paper \cite{Curtiss}, Curtiss showed a version of continuity theorem that, if $\{X_n\}_{n\in\bN}$ is a sequence of random variables in $\bR$ such that
\begin{itemize}
    \item[\textbf{(C)}] \label{item:curtiss_condition}
    $M_{X_{n}}(s)$ converges pointwise on a neighborhood of $s = 0$,
\end{itemize}
then $\{X_n\}_{n\in\bN}$ converge in distribution. This result has an advantage over L\'evy's continuity theorem in that the pointwise limit of $M_{X_n}$'s is guaranteed to be a m.g.f. Such a stronger conclusion is possible because \textbf{(C)} guarantees the tightness of $\{X_n\}_{n\in\bN}$.

In some applications, however, \textbf{(C)}~is quite costly to verify and requires extra inputs, while the stronger part of its conclusion - that the limit is always a m.g.f.\ of some distribution - is not essential. For instance, in~\cite{Kim} and \cite{KimLee}, the m.g.f.s of normalized descents in conjugacy classes are analyzed with their series expansions, which fail to converge for $s > 0$. In~\cite{Kim}, Kim circumvented this technical difficulty by establishing a bijection to show the convergence for $s>0$. In~\cite{KimLee}, Kim and Lee calculated an alternative form of the m.g.f.\ that is convergent for $s>0$ by expanding the original generating function in Laurent series at $\infty$ rather than at~$0$. Moreover, in both proofs, the limit is shown explicitly to be the m.g.f.\ of the normal distribution. If we were to take a similar approach, we would have to show that $M_{W_\lambda}$, the m.g.f.s of the normalized descent/major-index pairs $W_\lambda$, converge pointwise on an open set containing $(0,0)$. However, as the known series expansion of $M_{W_\lambda}$ is convergent only on $\left\{ (s,r) | s \leq 0, r \leq 0 \right\}$, we would need to use similar methods used in~\cite{KimLee} in order to come up with possibly several expressions for $M_{W_\lambda}$ that are convergent on different regions, whose union covers an open set containing $(0,0)$.


In this section, we provide a simple result that takes cares of this situation. We begin by introducing the following lemma, which is a not-so-famous entry of the big list of equivalent conditions for the convergence in distribution (known as Portmanteau theorem). A proof is also provided for self-containedness.

\begin{lemma}
	\label{lemma:conv_equivalence}
	Let $ X_{n} $ be random vectors in $ \bR^{d} $ for each $ n \in \bN \cup \{\infty\} $. Then the followings are equivalent:
	\begin{enumerate}
		\item $ X_{n} $ converges in distribution to $X_{\infty}$.
		\item $ X_{n} $ converges vaguely to $X_{\infty}$, i.e., $ \lim_{n\to\infty} \E[\varphi(X_{n})] = \E[\varphi(X_{\infty})] $ for all $ \varphi \in C_{c}(\bR^{d}) $, where $C_{c}(\bR^d) $ denotes the set of all continuous compactly supported functions $\varphi : \bR^d \to \bR$.
	\end{enumerate}
\end{lemma}

\begin{proof}
    By Portmanteau theorem, the condition (1) is equivalent to
    \begin{enumerate}
        \item[(1')] $ \lim_{n\to\infty} \E[\varphi(X_{n})] = \E[\varphi(X_{\infty})] $ for all $\varphi$ in the set $C_b(\bR^d)$ of all continuous bounded functions on $\bR^d$.
    \end{enumerate}
    This obviously implies (2). For the other direction, assume that (2) is true and fix a function $ \chi \in C_{c}(\bR^{d}) $ taking values in $[0, 1]$. Then, for each $ \varphi \in C_{b}(\bR^{d}) $,
	\begin{align*}
		\left\lvert \E[\varphi(X_{n})] - \E[\varphi(X_{\infty})] \right\rvert
		&\leq \left\lvert \E[\varphi(X_{n})\chi(X_{n})] -\E[\varphi(X_{\infty})\chi(X_{\infty})] \right\rvert \\
		&\quad + \|\varphi\|_{\sup}(1 - \E[\chi(X_{n})])+ \|\varphi\|_{\sup}(1 - \E[\chi(X_{\infty})]),
	\end{align*}
	where $ \|\varphi\|_{\sup} = \sup\{|\varphi(x)| : x\in\bR^{d}\} $. Taking limsup as $ n\to\infty $, vague convergence gives us
	\begin{align*}
		\limsup_{n\to\infty} \left\lvert \E[\varphi(X_{n})] -\E[\varphi(X_{\infty})] \right\rvert
		\leq 2 \|\varphi\|_{\sup}(1 - \E[\chi(X_{\infty})]).
	\end{align*}
	By the dominated convergence theorem, this bound vanishes as $ \chi \uparrow 1 $ pointwise. Hence, we have $ \lim_{n\to\infty}\E[\varphi(X_{n})] = \E[\varphi(X_{\infty})] $, and (1) follows.
\end{proof}

\begin{proposition}
	\label{thm:continuity}
	Let $ X_{n} $ be random vectors in $ \bR^{d} $ for each $ n \in \bN \cup \{\infty\} $. Suppose that there is a non-empty open subset $ U \subseteq \bR^{d} $ such that $ \lim_{n\to\infty} M_{X_{n}}(s) = M_{X_{\infty}}(s) $ for all $ s \in U $. Then, $ X_{n} $ converges in distribution to $ X_{\infty} $.
\end{proposition}

\begin{proof}
	Fix $ a \in U $ and introduce random vectors $ Y_{n} $ in $ \bR^{d} $ whose laws are given by the exponential tilting
	\begin{align*}
		\Prob[Y_{n} \in \mathrm{d}x] = \frac{e^{\langle a, x \rangle}}{M_{X_{n}}(a)} \, \Prob[X_{n} \in \mathrm{d}x].
	\end{align*}
	For each $ t \in \bR^{d} $, there exists $ \delta > 0 $ such that $ a + st \in U $ for all $ s \in (-\delta, \delta) $. Then, for all $s \in (-\delta,\delta)$, we have
	\[
		\lim_{n\to\infty} M_{\langle t, Y_{n}\rangle}(s)
		= \lim_{n\to\infty} \frac{M_{X_{n}}(a+st)}{M_{X_{n}}(a)}
		= \frac{M_{X_{\infty}}(a+st)}{M_{X_{\infty}}(a)}
		= M_{\langle t, Y_{\infty}\rangle}(s),
	\]
	and so, $ \langle t, Y_{n} \rangle $ converges to $ \langle t, Y_{\infty} \rangle $ in distribution by Curtiss' continuity theorem. By Cramer-Wold device, this implies that $ Y_{n} $ converges in distribution to $ Y_{\infty} $. Then, for each $ \varphi \in C_{c}(\bR^{d}) $, the function $ \varphi(\cdot)e^{\langle a, \cdot \rangle} $ is bounded, and so,
	\[
		M_{X_{n}}(a) \E\left[ \varphi(Y_{n}) e^{\langle a, Y_{n} \rangle} \right]
		\xrightarrow[n\to\infty]{} M_{X_{\infty}}(a) \E\left[ \varphi(Y_{\infty}) e^{\langle a, Y_{\infty} \rangle} \right].
	\]
	In other words, $\E[\varphi(X_n)] \to \E[\varphi(X_\infty)]$, and so, by Lemma \ref{lemma:conv_equivalence}, $ X_{n} $ converges to $ X_{\infty}$ in distribution.
\end{proof}


\section{Generating function of $(d(\pi), maj(\pi))$}

For each finite set $A$, we write $\lvert A \rvert$ for the cardinality of $A$. For each integers $a \leq b$, the double-struck interval notation $\llbracket a, b \rrbracket = [a, b] \cap \bZ $ denotes the set of all integers between $a$ and $b$.

Throughout this article, $ \lambda $ will always denote an integer partition of a non-negative integer $ n $. For each permutation $\pi \in S_n$, we write $m_k(\pi)$ for the number of $k$-cycles in $\pi$. Then, $\cC_{\lambda} = \{ \pi \in S_n : m_k(\pi) = \lambda_k \text{ for all } k \in \llbracket 1, n \rrbracket \}$ denotes the conjugacy class of $ S_{n} $ with the cycle structure $\lambda$. Associated to each $ \lambda $ is the density $ \alpha_{1} = \alpha_1(\lambda) = \lambda_{1}/n $ of fixed points. We will see that $ \alpha_{1} $ is essentially the only parameter that determines the shape of the limiting distribution of the descent/major index pair.

We will set up some probability notations. For each integer partition $\lambda$, $\Prob_{\lambda}$ will denote the probabiliy law under which $\pi$ has the uniform distribution over $\cC_{\lambda}$ and $\sigma_i$ has the uniform distribution over $S_{\lambda_i}$ for each $i \in \llbracket 1, n \rrbracket $. We will always assume, without mentioning, that $\pi$ and $\sigma_i$'s have their respective, aforementioned distributions under $\Prob_{\lambda}$. Then, $\E_{\lambda}$ will denote the expectation corresponding to $\Prob_{\lambda}$.

It is convenient to consider some special functions. Let $\Gamma(s)$ denote the gamma function. By setting $x! = \Gamma(x+1)$, the factorial extends to all of $ \bR \setminus \{-1, -2, \cdots\}$. Then, we extend binomial coefficients accordingly. We also introduce the $q$-bracket notation $[a]_q = \frac{1-q^a}{1-q}$.

In \cite[Theorem 1]{Fulman1}, Fulman derived a formula for the generating function of descent numbers in conjugacy classes. This formula was a key ingredient in~\cite{Kim} and~\cite{KimLee} for establishing the central limit theorems. In this section, we would like to derive an analogous formula for the generating function of pairs of descent/major-index in conjugacy classes, which is the statement of the following proposition.

\begin{proposition}
	Let $ f_{i,a}(q) = \frac{1}{i} \sum_{d \mid i} \mu(d) [a]_{q^{d}}^{i/d} $. If $\sigma_i$ has the uniform distribution over $S_{\lambda_i}$ under $\Prob_{\lambda}$ for each $i \in \llbracket 1, n \rrbracket $, then,
	\begin{equation}
		\label{eqn:genfn_master}
		\frac{ \sum_{\pi \in \cC_{\lambda}} t^{d(\pi)}q^{maj(\pi)} }{(1-t)(1-qt)\cdots(1-q^{n}t)}
		= \sum_{a \geq 1} t^{a} \left( \prod_{i=1}^{n} \E_{\lambda} \left[ \prod_{k \geq 1} f_{i,a}(q^{k})^{m_{k}(\sigma_{i})} \right] \right).
	\end{equation}
\end{proposition}

\begin{proof}
	Recall that for non-negative integers $r_{1}, \cdots, r_{n}$ summing to $n$, the quantity
	\begin{align*}
		M(r_1,\cdots,r_a)
		= \frac{1}{n} \sum_{d\mid r_{1},\cdots,r_{a}} \mu(d) \frac{(n/d)!}{(r_{1}/d)!\cdots(r_{a}/d)!}
	\end{align*}
	counts the number of primitive circular words of length $ n $ from the alphabet $ \llbracket 1, a \rrbracket $ in which the letter $ i $ appears $ r_{i} $ times. Associated to this quantity, we define $ J_{i,m,a} $ by
	\begin{align*}
		J_{i,m,a} = \sum_{\substack{ r_{1}+\cdots+r_{a} = i \\ 1r_{1} + \cdots + ar_{a} = i+m }} M(r_{1},\cdots,r_{a}).
	\end{align*}
	Then, one of the main results in \cite{Fulman1} is the generating function
	\begin{align}
		\label{eqn:genfn_fulman}
		\sum_{n\geq 0} \frac{y^{n}\sum_{\pi \in S_{n}} t^{d(\pi)}q^{maj(\pi)} \prod_{i} x_{i}^{m_{i}(\pi)}}{(1-t)(1-qt)\cdots(1-q^{n}t)}
		= \sum_{a\geq 1} t^{a} \prod_{\substack{i \geq 1 \\ m \geq 0}} \left( \frac{1}{1 - q^{m}x_{i}y^{i}} \right)^{J_{i,m,a}}.
	\end{align}
	Our strategy is to expand the huge product on the right-hand side. In the course of computations, the following lemmas will be useful.
	
	\begin{lemma}
		\label{lemma:fiaq_formula}
		For $q \in [0, 1]$, we have $ f_{i,a}(q) = \sum_{m \geq 0} J_{i,m,a}q^{m} $.
	\end{lemma}

	\begin{proof}[Proof of Lemma \ref{lemma:fiaq_formula}]
		Plugging in the definition of $ J_{i,m,a} $ and rearranging, we see that
		\begin{align*}
			\sum_{m \geq 0} J_{i,m,a}q^{m}
			&= \frac{1}{i} \sum_{d \mid i} \mu(d) \sum_{m \geq 0} q^{m} \sum_{\substack{ r_{1}+\cdots+r_{a} = i \\ 1r_{1} + \cdots + ar_{a} = i+m \\ d\mid r_{1},\cdots,r_{a} }} \frac{(i/d)!}{(r_{1}/d)!\cdots(r_{a}/d)!} \\
			&= \frac{1}{i} \sum_{d \mid i} \mu(d) \sum_{\substack{ r_{1}+\cdots+r_{a} = i \\ d\mid r_{1},\cdots,r_{a} }} \frac{(i/d)!}{(r_{1}/d)!\cdots(r_{a}/d)!} q^{\sum_k (k-1)r_k}.
		\end{align*}
		Noting that $ \frac{i!}{r_{1}!\cdots r_{a}!} $ with $ r_{1} + \cdots + r_{a} = i $ is the number of tuples $ (x_{1}, \cdots, x_{i}) \in \llbracket 1, a \rrbracket^{i} $ in which $ k $ appears $ r_{k} $ times, we can rewrite the above formula as
		\begin{align*}
			\sum_{m \geq 0} J_{i,m,a}q^{m}
			&= \frac{1}{i} \sum_{d \mid i} \mu(d) \sum_{1 \leq x_{1},\cdots,x_{i/d} \leq a} q^{d\sum_{k=1}^{i/d} (x_{k} - 1)},
		\end{align*}
		which yields the definition of $ f_{i,a} (q)$ from the formula $\sum_{x=1}^{a} q^{d(x-1)} = [a]_{q^d}$.
	\end{proof}

	\begin{lemma}
		\label{lemma:genfn_num_cycles}
		Suppose that $\sigma$ is uniformly distributed over $S_n$. Then,
		\begin{align}
		    \label{eqn:genfn_num_cycles}
		    \E \left[\prod_{k} x_{k}^{m_{k}(\sigma)}\right] = \sum_{r \vdash n} \left[ \prod_{k} \frac{1}{r_{k}!} \left( \frac{x_{k}}{k} \right)^{r_{k}} \right].
		\end{align}
	\end{lemma}
	
	\begin{proof}[Proof of Lemma \ref{lemma:genfn_num_cycles}]
		For each $ r \vdash n $, the number of permutations in $ S_{n} $, for which the number of $ k $-cycles is exactly $ r_{k} $, is given by $ n!/\prod_{k} r_{k}!k^{r_{k}} $. Plugging this into
		\begin{align*}
		    \E\left[\prod_{k} x_{k}^{m_{k}(\sigma)}\right] = \sum_{r \vdash n} \left[ \prod_k x_k^{r_k} \right] \Prob\left[ m_k(\sigma) = r_k \text{ for all } k = 1, \cdots, n \right]
		\end{align*}
		proves the desired identity.
	\end{proof}
	
	Returning to the proof of Proposition \ref{eqn:genfn_master}, we find that the coefficient of the right-hand side of \eqref{eqn:genfn_fulman} can be rearranged, by using the series expansion $-\log(1-x) = \sum_{k=1}^{\infty} \frac{1}{k}x^k$ and Lemma \ref{lemma:fiaq_formula}, as
	\begin{align}
		\prod_{\substack{i \geq 1 \\ m \geq 0}} \left( \frac{1}{1 - q^{m}x_{i}y^{i}} \right)^{J_{i,m,a}}
		= \exp \left\{ \sum_{i \geq 1}\sum_{k \geq 1} \frac{1}{k}f_{i,a}(q^{k}) x_{i}^{k} y^{ik} \right\}.
		\tag{$\diamond$}
	\end{align}
	For each given $ i $ and $ k $, we apply the Taylor series $ e^{x} = \sum_{s = 0}^{\infty} \frac{x^{s}}{s!} $ to expand the factor $ \exp \{ \frac{1}{k}f_{i,a}(q^{k}) x_{i}^{k} y^{ik} \} $. Since the generic variable $ s $ needs to be distinguished for different choices of $ i $ and $ k $, we explicate this dependence by writing the generic variables as $ s_{i,k} $. The resulting expansion takes the form
	\begin{align*}
		(\diamond) = \sum_{(s_{i,k})_{i,k\in\bN}} \prod_{i,k \geq 1}\frac{1}{s_{i,k}!} \left( \frac{f_{i,a}(q^{k}) x_{i}^{k} y^{ik}}{k} \right)^{s_{i,k}}.
	\end{align*}
	Now, for each $\lambda \vdash n$, we collect all terms satisfying $\sum_{k} k s_{i,k} = \lambda_i$ for each $i \in \llbracket 1, n \rrbracket $. Then, by Lemma \ref{lemma:genfn_num_cycles}, $(\diamond)$ simplifies to
	\begin{align*}
		(\diamond)
		&= \sum_{n\geq 0} \sum_{\lambda \vdash n} \sum_{\substack{(s_{i,k})_{i,k\in\bN} \\ \sum_{k} k s_{i,k} = \lambda_i}} \left[ \prod_{i,k \geq 1}\frac{1}{s_{i,k}!} \left( \frac{f_{i,a}(q^{k})}{k} \right)^{s_{i,k}} \right] y^n \prod_{i\geq 1} x_i^{\lambda_i} \\
		&\stackrel{\text{\eqref{eqn:genfn_num_cycles}}}= \sum_{n \geq 0} \sum_{\lambda \vdash n} \left( \prod_{i=1}^{n} \E_{\lambda} \left[ \prod_{k \geq 1} f_{i,a}(q^{k})^{m_{k}(\sigma_{i})} \right] \right)  y^{n} \prod_{i \geq 1} x_{i}^{\lambda_{i}} .
	\end{align*}
	Plugging this back into the generating function proves \eqref{eqn:genfn_master} as required.
\end{proof}

As a sanity check, recall that $ \E \left[ x^{\sum_{i} m_{i}(\sigma)} \right] = \binom{n + x - 1}{n} $ holds if $\sigma$ is uniformly distributed over $S_n$. Taking the limit as $ q \uparrow 1 $ to the key identity \eqref{eqn:genfn_master}, we obtain
\begin{align*}
	\frac{ \sum_{\pi \in \cC_{\lambda}} t^{d(\pi)} }{(1-t)^{n+1}}
	= \sum_{a \geq 1} t^{a} \prod_{i \geq 1} \binom{\lambda_{i} + f_{i,a}(1) - 1}{\lambda_{i}},
\end{align*}
which is exactly the conclusion of \cite[Corollary 3]{Fulman1}.

In the authors' paper \cite{KimLee}, it was useful to obtain an estimation for the quantity $ f_{i,a}(1) $. Likewise, an analogous estimation for $ f_{i,a}(q) $ will be useful in dealing with \eqref{eqn:genfn_master}. The following lemma serves this purpose.

\begin{lemma}
    \label{lemma:fiaq_est}
	We have $ [a]_{q}^{i} - \frac{i}{2} [a]_{q}^{i/2} \leq i f_{i,a}(q) \leq [a]_{q}^{i} $ for any $ q \in [0, 1] $.
\end{lemma}

\begin{proof}
	First, we note that $ M(r_{1}, \cdots, r_{a}) $ is bounded from above by the number of circular words of length $ i = r_{1} + \cdots + r_{a} $ from the alphabet $ \llbracket 1, a \rrbracket $ in which the letter~$ k $ appears $ r_{k} $ times, which is exactly $ \frac{1}{i} \frac{i!}{r_{1}!\cdots r_{a}!} $. This gives the bound
	\begin{align*}
		f_{i,a}(q)
		&\leq \frac{1}{i} \sum_{m \geq 0} q^{m} \sum_{\substack{ r_{1}+\cdots+r_{a} = i \\ 1r_{1} + \cdots + ar_{a} = i+m }} \frac{i!}{r_{1}!\cdots r_{a}!}.
	\end{align*}
	Following the idea of the proof of Lemma \ref{lemma:fiaq_formula}, we find that this bound equals the desired upper bound. For the lower bound, we use the following crude estimate
	\begin{align*}
		i f_{i,a}(q)
		\geq [a]_{q}^{i} - \sum_{\substack{d \mid i \\ d \neq i}} [a]_{q^{i/d}}^{d}
		\geq [a]_{q}^{i} - \sum_{\substack{d \mid i \\ d \neq i}} [a]_{q}^{i/2}
		\geq [a]_{q}^{i} - \frac{i}{2} [a]_{q}^{i/2}.
	\end{align*}
	Here, the second inequality follows from the fact that $ [a]_{q^{k}} $ is decreasing in $ k $.
\end{proof}


\section{Main result}

Let $\pi$ be chosen uniformly at random from $\cC_{\lambda}$. In order to establish the asymptotic normality of $(d(\pi), maj(\pi))$, we consider the following normalization
\begin{align*}
    W_{\lambda} = \left( \frac{d(\pi) - \frac{1-\alpha_1^2}{2}n}{n^{1/2}}, \frac{maj(\pi) - \frac{1-\alpha_1^2}{4}n^2}{n^{3/2}} \right).
\end{align*}
We aim to prove that $W_{\lambda}$ is asymptotically normal with mean zero and covariance matrix $\Sigma_{\alpha_1}$, where $\Sigma_{\alpha}$ is defined by the following $2$ by $2$ matrix
\begin{align*}
    \Sigma_{\alpha}
    = \begin{pmatrix}
		\frac{1}{12}(1-4\alpha^{3}+3\alpha^{4}) & \frac{1}{24}(1-4\alpha^{3}+3\alpha^{4}) \\
		\frac{1}{24}(1-4\alpha^{3}+3\alpha^{4}) & \frac{1}{36}(1-\alpha^{3})
	\end{pmatrix}.
\end{align*}
Since any real symmetric matrix determines a quadratic form and vice versa, we will abuse the notation to write $\Sigma_{\alpha}(x) = x^{\mathsf{T}} \Sigma_{\alpha} x$ for any $x \in \bR^2$. The goal of this section is to establish the proof of the following uniform estimate.

\begin{theorem}
    \label{thm:unifest}
    For each $s > 0$ and $r > 0$, there exists a constant $C = C(r, s) > 0$, depending only on $s$ and $r$, such that
    \begin{align}
        \label{eqn:unifest}
        \left\lvert M_{W_{\lambda}}(-s, -r) - e^{\frac{1}{2}\Sigma_{\alpha_1}(s, r)} \right\rvert \leq C(r, s) n^{-1/6}
    \end{align}
    holds for any $n \geq 1$ and for any conjugacy class $\cC_{\lambda}$ of $S_n$.
\end{theorem}

\subsection{Notations and conventions}

For the remainder of this paper, we fix two positive reals $s, r > 0$. It will become  clear that the window of scale $r/n^{3/2}$ is a natural choice for analyzing the behavior of $W_{\lambda}$. For brevity's sake, we write
\begin{align*}
    \delta = r/n^{3/2}.
\end{align*}
Comparing the m.g.f.\ of $ W_{\lambda}$ to the generating function \eqref{eqn:genfn_master} shows that $q$ and $t$ are related to $n$ by $ q = e^{-r/n^{3/2}} = e^{-\delta}$ and $ t = e^{-s/n^{1/2}} = e^{-\frac{sn}{r}\delta}$, and we assume so hereafter. We also choose $ \epsilon > 0$ such that $2 e (s+r)\epsilon / r < 1$.

In what follows, the asymptotic notations $ f(x) = \cO_{a}(g(x)) $ and $f(x) \lesssim_a g(x)$ will denote the fact that there exists a constant $C > 0$, depending only on $s$, $r$ and the parameter $a$, such that $ |f(x)| \leq C g(x) $ holds for all $ x $ in the prescribed range. If no parameter $a$ is involved, we simply write $f(x) = \cO(g(x))$ or $f(x) \lesssim g(x)$.

Along the proof, we will encounter large chunks of expressions. Since it is counter-productive and not aesthetic to carry them all the way through the proof, we will introduce generic symbols to replace them. First, we define $\mK_{\lambda, r,a,i}$ and $\mL_{\lambda, s, r}$ by
\begin{align*}
    \mK_{\lambda, r,a,i}
    := \lambda_{i}! i^{\lambda_{i}} \E_{\lambda} \left[ \prod_{k\geq 1} f_{i,a}(q^{k})^{m_{k}(\sigma_{i})} \right],
\end{align*}
and
\begin{align*}
    \mL_{\lambda, s, r}
    := \frac{ t^{\frac{\alpha_1^2}{2}n}q^{\frac{\alpha_1^2}{4}n^2} }{ \frac{1}{\delta^{n+1}} \int_{0}^{1} u^{\frac{sn}{r}-1}(1-u)^{n} \, \mathrm{d}u } \left( \sum_{a \geq 1} t^{a} \left( \prod_{i=1}^{n} \mK_{\lambda,r,a,i} \right) \right).
\end{align*}
For the definition of $\mK_{\lambda, r,a,i}$, we recall that $\sigma_i$ has the uniform distribution over $S_{\lambda_i}$ under~$\Prob_{\lambda}$ for each $i \in \llbracket 1, n \rrbracket $. Also, for convenience, we decompose~$\mL$ further into
\begin{align*}
    \mL_{\lambda, s, r} = \mL_{\mathrm{small},\lambda, s, r} + \mL_{\mathrm{large},\lambda, s, r},
\end{align*}
where $\mL_{\mathrm{small},\lambda, s, r}$ (respectively, $\mL_{\mathrm{large},\lambda, s, r}$) is the restriction of the sum in the definition of $\mL$ onto the range $a < \epsilon / \delta$ (respectively, $a \geq \epsilon / \delta$). Then, we define $\mF_{n,r,k}(u)$ and $\mG_{\lambda,s,r}(u)$ by
\begin{align*}
    \mF_{n,r,k}(u)
    := \frac{\delta^{k-1}}{k^2} \frac{1-u^k}{(1-u)^k},
\end{align*}
and
\begin{align*}
    \mG_{\lambda,s,r}(u)
    := t^{\frac{\alpha_1^2}{2}n}q^{\frac{\alpha_1^2}{4}n^2} \lambda_1! \sum_{\mu \vdash \lambda_1} \prod_{k\geq 1} \frac{\mF_{n,r,k}(u)^{\mu_k}}{\mu_k!}.
\end{align*}
The roles of these quantities will become clear as the proof proceeds. Finally, for these quantities, $s$, $r$ and $\lambda$ will be suppressed notationally whenever the dependence on these variables is clear from context.

\subsection{Separating the contribution of fixed points}

In this section, we provide a representation of the m.g.f.\ of $W_{\lambda}$ which is adequate for analyzing the effect of fixed points. Rewrite $M_{W_{\lambda}}$ as
\begin{align}
    \label{eqn:mgf_expand}
    M_{W_{\lambda}}(-s, -r)
    = t^{-\frac{1-\alpha_1^2}{2}n}q^{-\frac{1-\alpha_1^2}{4}n^2} \E_{\lambda}\left[ t^{d(\pi)} q^{maj(\pi)} \right],
\end{align}
where we recall that $\pi$ is uniformly distributed over $\cC_{\lambda}$ under $\Prob_{\lambda}$ and that $q$ and $t$ are defined by $q = e^{-r/n^{3/2}}$ and $t = e^{-s/n^{1/2}}$. To prevent the reader from being distracted by the jumble of computations, we first state the main result of this section. This is a combination of Propositions \ref{prop:est_common_factor}, \ref{prop:est_l_small}, and \ref{prop:est_l_large}.

\begin{proposition}
    \label{prop:mgf_expand_2}
    As $n\to\infty$, we have
    \begin{align*}
        M_{W_{\lambda}}(-s, -r)
        = \left(1 + \cO\big(n^{-1/2}\big)\right) e^{\frac{1}{2}\Sigma_{0}(s, r)} \left( \mL_{\mathrm{large}} + \mL_{\mathrm{small}}  \right),
    \end{align*}
    where $\mL_{\mathrm{small}}$ decays at least exponentially fast and $\mL_{\mathrm{large}}$ is asymptotically the ratio of two integrals
    \begin{align*}
    	\mL_{\mathrm{large}}
        = \left(1 + \cO\big(n^{-1/2}\big)\right)
        \frac{ \int_{0}^{e^{-\epsilon}} u^{\frac{sn}{r}-1}(1-u)^{n} \mG(u) \, \mathrm{d}u }{ \int_{0}^{1} u^{\frac{sn}{r}-1}(1-u)^{n} \, \mathrm{d}u }.
    \end{align*}
\end{proposition}

Returning to the problem of estimating \eqref{eqn:mgf_expand}, we will adopt \eqref{eqn:genfn_master} as our starting point. Since $\lvert t \rvert < 1$ and $\lvert q \rvert < 1$, the generating function \eqref{eqn:genfn_master} converges absolutely. We first analyze the asymptotic behavior of the common factor $\prod_{j=0}^{n} (1-tq^{j})$.

\begin{proposition}
    \label{prop:est_common_factor}
    As $n\to\infty$, we have
    \begin{align}
        \label{eqn:est_common_factor}
    	\frac{1}{n!} \prod_{j=0}^{n} (1-tq^{j})
    	= \left(1 + \cO\big(n^{-1/2}\big)\right)
    	\frac{ t^{\frac{1}{2}n}q^{\frac{1}{4}n^2} e^{\frac{1}{2}\Sigma_{0}(s, r)} }{ \frac{1}{\delta^{n+1}} \int_{0}^{1} u^{\frac{sn}{r}-1}(1-u)^{n} \, \mathrm{d}u }.
    \end{align}
    Consequently,
    \begin{align}
        \label{prop:mgf_expand_3}
        M_{W_{\lambda}}(-s, -r) = \left(1 + \cO\big(n^{-1/2}\big)\right) e^{\frac{1}{2}\Sigma_{0}(s, r)} \mL.
    \end{align}
\end{proposition}

\begin{proof}[Proof of Proposition \ref{prop:est_common_factor}]
    Plugging the generating function \eqref{eqn:genfn_master} into the m.g.f. \eqref{eqn:mgf_expand} and utilizing the definition of $\mK_{a,i}$, we see that
    \begin{align*}
        M_{W_{\lambda}}(-s, -r)
        &= t^{-\frac{1-\alpha_1^2}{2}n}q^{-\frac{1-\alpha_1^2}{4}n^2} \left( \prod_{j=0}^{n} (1-tq^{j}) \right) \frac{1}{\lvert \cC_{\lambda} \rvert} \sum_{a \geq 1} t^{a} \left( \prod_{i=1}^{n} \frac{\mK_{a,i}}{\lambda_i! i^{\lambda_i}} \right) \\
        &= t^{-\frac{1-\alpha_1^2}{2}n}q^{-\frac{1-\alpha_1^2}{4}n^2} \left( \frac{1}{n!} \prod_{j=0}^{n} (1-tq^{j}) \right) \sum_{a \geq 1} t^{a} \left( \prod_{i=1}^{n} \mK_{a,i} \right),
    \end{align*}
    where the identity $\lvert \cC_{\lambda} \rvert = n!/\prod_i \lambda_i!i^{\lambda_i}$ is used in the second equality. Then, in view of the definition of $\mL$, \eqref{eqn:est_common_factor} will imply \eqref{prop:mgf_expand_3}. So it remains to prove \eqref{eqn:est_common_factor}.
    
    Write $g(x) = \log\big( \frac{1 - e^{-x}}{x} \big)$ and $x_j = - \log\left( t q^j \right)$. Then,
    \begin{align*}
        \frac{1}{n!} \prod_{j=0}^{n} (1-tq^{j})
        &= \exp\left\{ \sum_{j=0}^{n} g\left( x_j \right) \right\} \frac{1}{n!} \prod_{j=0}^{n} x_j.
    \end{align*}
    Note that $x_j = \frac{s}{n^{1/2}} + \frac{jr}{n^{3/2}} = \cO(n^{-1/2})$ uniformly in $0 \leq j \leq n$. Plugging this to the Taylor expansion $ g(x) = -\frac{1}{2}x + \frac{1}{24}x^2 + \cO(x^3)$ near $x = 0$, we obtain
    \begin{align*}
        \sum_{j=0}^{n} g\left( x_j \right)
        &= \sum_{j=0}^{n} \left[ -\frac{1}{2} \left(\frac{s}{n^{1/2}} + \frac{jr}{n^{3/2}}\right) + \frac{1}{24} \left( \frac{s}{n^{1/2}} + \frac{jr}{n^{3/2}} \right)^2 + \cO\big(n^{-3/2}\big) \right] \\
        &= - \frac{s n^{1/2}}{2} - \frac{rn^{1/2}}{4} + \frac{1}{2}\Sigma_0(s, r) + \cO\big(n^{-1/2}\big).
    \end{align*}
    Exponentiating both sides produces the numerator of \eqref{eqn:est_common_factor} as well as the desired relative error of the magnitude $\cO\big(n^{-1/2}\big)$. For the denominator of \eqref{eqn:est_common_factor}, we have
    \begin{align}
        \label{eqn:est_02}
        \begin{split}
            \frac{1}{n!} \prod_{j=0}^{n} x_j
            &= \frac{1}{n!} \prod_{j=0}^{n} \left( \left( \frac{n s}{r} + j \right)\delta \right)
             = \delta^{n+1} \frac{ (\frac{sn}{r}+n)! }{ (\frac{sn}{r}-1)!n! } \\
            &= \left[ \frac{1}{\delta^{n+1}} \int_{0}^{1} u^{\frac{sn}{r}-1}(1-u)^{n} \, \mathrm{d}u \right]^{-1},
        \end{split}
    \end{align}
    where we exploited the beta function identity $\int_{0}^{1} u^a (1-u)^ b \, \mathrm{d}u = \frac{a!b!}{(a+b+1)!}$, which holds for any reals $a, b > -1$, in the last step.
\end{proof}

Proposition \ref{prop:est_common_factor} illuminates the meaning of $\mL$ as the normalized m.g.f. Next, we make a simple observation on $\mK_{a,i}$ to be used later.

\begin{lemma}
    \label{lemma:est_kai_elem}
    We have
    \begin{align}
        \label{eqn:est_kai_elem}
        (if_{i,a}(q))^{\lambda_i}
        \leq \mK_{a,i}
        \leq \lambda_{i}! i^{\lambda_{i}} \binom{\lambda_{i} + f_{i,a}(q) - 1}{\lambda_{i}}
        \leq (if_{i,a}(q))^{\lambda_i} e^{\lambda_{i}^{2}/f_{i,a}(q)}.
    \end{align}
\end{lemma}

\begin{proof}[Proof of Lemma \ref{lemma:est_kai_elem}]
    The lower bound is easily obtained by focusing on the contribution of $\sigma_i$'s with $m_1(\sigma_i) = \lambda_i$. Indeed, since $m_1(\sigma_i) = \lambda_i$ is achieved only by the identity permutation in $S_{\lambda_i}$, we have $\Prob_{\lambda}[m_1(\sigma_i) = \lambda_i] = \frac{1}{\lambda_i!}$, and so,
    \begin{align*}
        \mK_{a,i}
        \geq \lambda_i! i^{\lambda_i} f_{i,a}(q)^{\lambda_{i}} \Prob_{\lambda}[m_1(\sigma_i) = \lambda_i]
        = (if_{i,a}(q))^{\lambda_i}.
    \end{align*}
    Now, we focus on the upper bound. In view of Lemma \ref{lemma:fiaq_formula}, $f_{i,a}(q^k) \leq f_{i,a}(q)$, and so, by Lemma \ref{lemma:genfn_num_cycles}, we have
    \begin{align*}
        \mK_{i,a}
        \leq \lambda_{i}! i^{\lambda_{i}} \E_{\lambda}\left[ f_{i,a}(q)^{\sum_k m_{k}(\sigma_{i})} \right]
        = \lambda_{i}! i^{\lambda_{i}} \binom{\lambda_{i} + f_{i,a}(q) - 1}{\lambda_{i}}.
    \end{align*}
    Therefore, \eqref{eqn:est_kai_elem} follows from the inequality $\binom{a+b-1}{a} \leq \frac{b^a}{a!}e^{a^2/b}$.
\end{proof}

This easily leads to the following crude but useful estimate.

\begin{proposition}
    \label{prop:est_l_small}
    As $n\to\infty$, we have
    \begin{align}
        \label{eqn:est_l_small}
    	\mL_{\mathrm{small}}
    	\lesssim n^{1/2} \left( 2 e(s+r)\epsilon/r \right)^{n+1}.
    \end{align}
\end{proposition}

\begin{proof}[Proof of Proposition \ref{prop:est_l_small}]
    Let $ a < \epsilon / \delta $. We first consider the case $ n \leq \epsilon / \delta$. By Lemma \ref{lemma:fiaq_est}, we have $if_{i,a}(q) \leq [a]_q^i \leq a^i$, and so,
    \begin{align*}
        \lambda_i! i^{\lambda_i} \binom{\lambda_i + f_{i,a}(q) - 1}{\lambda_i}
        \leq (i\lambda_i + if_{i,a}(q))^{\lambda_i}
        \leq (i\lambda_i + a^i)^{\lambda_i}
    \end{align*}
    Then, since $i\lambda_i + a^i \leq n + a^i \leq (2\epsilon / \delta)^i$, we see that
    \begin{align*}
    	\prod_{i=1}^{n} \mK_{a,i}
    	\stackrel{\text{\eqref{eqn:est_kai_elem}}}\leq \prod_{i=1}^{n} \lambda_{i}! i^{\lambda_{i}} \binom{\lambda_{i} + f_{i,a}(q) - 1}{\lambda_{i}}
    	\leq \prod_{i=1}^{n} (i \lambda_{i} + a^{i})^{\lambda_{i}}
    	\leq (2\epsilon / \delta)^n.
    \end{align*}
    By taking the union bound, we have
    \begin{align*}
        \sum_{1 \leq a < \epsilon / \delta} t^a \left( \prod_{i=1}^{n} \mK_{a,i} \right)
        \leq (\epsilon /\delta) \max_{a < \epsilon / \delta} \left( \prod_{i=1}^{n} \mK_{a,i} \right)
        \leq (2\epsilon / \delta)^{n+1}.
    \end{align*}
    Also, by Stirling's approximation and the trivial bound, $t^{\frac{\alpha_1^2}{2}n}q^{\frac{\alpha_1^2}{4}n^2} \leq 1$, we get
    \begin{align*}
        \frac{ t^{\frac{\alpha_1^2}{2}n}q^{\frac{\alpha_1^2}{4}n^2} }{ \frac{1}{\delta^{n+1}} \int_{0}^{1} u^{\frac{sn}{r}-1}(1-u)^{n} \, \mathrm{d}u }
        &\stackrel{\text{\eqref{eqn:est_02}}}\leq \frac{1}{n!} \prod_{j=0}^{n} \left( \left( \frac{n s}{r} + j \right)\delta\right) \\
        &\leq \frac{1}{n!} \left( \frac{s+r}{r}n\delta \right)^{n+1}
        \lesssim n^{1/2} \left( e(s+r)\delta/r \right)^{n+1}.
    \end{align*}
    Combining these altogether, it follows that \eqref{eqn:est_l_small} holds for $n \leq \epsilon / \delta$. Since there are only finitely many $\lambda$'s for which $n > \epsilon / \delta$, we may absorb their contributions to the implicit bound, and thus, \eqref{eqn:est_l_small} holds unconditionally.
\end{proof}

Our next goal is to estimate $\mL_{\mathrm{large}}$. From now on, we focus on the case $ a \geq \epsilon / \delta $. A trivial but crucial observation is that $0 \leq q^a \leq q^{\epsilon/\delta} = e^{-\epsilon}$ holds, hence $q^a$ is uniformly away from $1$. This fact will be extensively used in the sequel, and is indeed one of the main reasons for separating the small range, $a \leq \epsilon / \delta$, from our main computation. We begin by improving the estimate on $\mK_{i,a}$'s.

\begin{lemma}
    \label{lemma:est_kai_detailed}
    For any integer $a \geq \epsilon / \delta$ and for any real $x$ satisfying $x \geq \epsilon/\delta$ and $|x - a| \leq 1$, we have
    \begin{align}
        \label{eqn:est_kai_detailed}
        \prod_{i\geq 1} \mK_{a,i}
        = \frac{1 + \cO\big(n^{-1/2}\big)}{\delta^n} \left( 1-q^x  \right)^{n} \lambda_1 ! \E_{\lambda} \left[ \prod_{k\geq 1} (k \mF_k(q^x))^{m_k(\sigma_1)} \right]
    \end{align}
\end{lemma}

\begin{proof}[Proof of Lemma \ref{lemma:est_kai_detailed}]
    Pick $N \in \bN$ so that $N^{3/2} \geq \frac{r}{1-e^{-\epsilon}}$ and assume that $n \geq N$. Let $x \in \bR$ be such that $x \geq \epsilon/\delta$ and $|x - a| \leq 1$. Then, for $k \in \llbracket 1, n \rrbracket$,
    \begin{align}
        \label{eqn:est_01}
        \begin{split}
        \frac{[a]_{q^k}}{(1-q^{kx})/k\delta}
        = \frac{k\delta}{1-e^{k\delta}}\left( 1 + \frac{q^{kx}}{1-q^{kx}}\left(1 - q^{k(a-x)}\right) \right)
        = 1 + \cO(k \delta).
        \end{split}
    \end{align}
    In particular, there exists a constant $c > 0$ such that $[a]_q \geq e^{-c\delta} \cdot \frac{1 - q^x}{\delta}$. Then, by our choice of $N$, we have $\frac{1-q^x}{\delta} \geq \frac{1-e^{-\epsilon}}{\delta} \geq 1$. Thus, for each $0 \leq p_0 \leq p \leq n$,
    \begin{align}
        \label{eqn:est_012}
        [a]_q^p
        \geq e^{-cp\delta} \left( \frac{1 - q^x}{\delta} \right)^p
        \geq e^{-cr} \left( \frac{1 - e^{-\epsilon}}{\delta} \right)^{p_0}.
    \end{align}
    The purpose of this explicit estimate is to reassure that the bound in $[a]_q^p \gtrsim_{p_0} \delta^{-p_0}$ depends only on $p_0$, $s$ and $r$. This will be crucial for producing uniform estimates in the next step.
    
    With these estimations at hand, we examine $\prod_{i\geq 2} \mK_{a,i}$ first. Let $i \in \llbracket 2, n \rrbracket$. By Lemma~\ref{lemma:fiaq_est}, \eqref{eqn:est_012} with $(p_0, p) = \left(1, \frac{i}{2}\right)$, and \eqref{eqn:est_01} with $k = 1$,
    \begin{align*}
        if_{i,a}(q)
        = [a]_q^i \left( 1 + \cO\left( \frac{i}{[a]_q^{i/2}} \right)\right)
        \stackrel{\text{\eqref{eqn:est_012}}}= [a]_q^i \left( 1 + \cO( i \delta )\right)
        \stackrel{\text{\eqref{eqn:est_01}}}= \left( \frac{1-q^x}{\delta} \right)^i e^{\cO(i\delta )}.
    \end{align*}
    Then, by \eqref{eqn:est_012} with $(p_0, p) = (2, i)$ and $i\delta \lesssim n^{-1/2}$, we have $if_{i,a}(q) \gtrsim [a]_q^i \gtrsim \delta^{-2}$, and hence $\frac{1}{f_{i,a}(q)} \lesssim i\delta^2$. So,
    \begin{align*}
        \prod_{i=2}^{n} \mK_{a,i}
        &\stackrel{\text{\eqref{eqn:est_kai_elem}}}= \prod_{i=2}^{n} \left[ (if_{i,a}(q))^{\lambda_i} e^{\cO\left( \lambda_i^2 / f_{i,a}(q) \right)} \right] \\
        &\quad = \prod_{i=2}^{n} \left[ \left( \frac{1-q^x}{\delta} \right)^{i\lambda_i}
        e^{\cO(i\lambda_i \delta ) + \cO(i\lambda_i^2 \delta^2 )}
        \right] \\
        &\quad = \left( \frac{1-q^x}{\delta} \right)^{n-\lambda_1} e^{\cO(n\delta) + \cO(n^2\delta^2)},
    \end{align*}
    where the last line follows from both $\sum_{i=2}^{n} i \lambda_i = n - \lambda_1 \leq n$ and $\sum_{i=2}^{n} i \lambda_i^2 \leq n^2$. Since $n\delta \lesssim n^{-1/2}$ and $n^2\delta^2 \lesssim n^{-1}$, the contribution of the factor $e^{\cO(n\delta) + \cO(n^2\delta^2)}$ is exactly $1 + \cO\big(n^{-1/2}\big)$ as required.
    
    As for $\mK_{a,i}$, we plug \eqref{eqn:est_01} directly into the definition of $\mK_{a,1}$. Then,
    \begin{align*}
        \mK_{a,1}
        &= \lambda_1! \E_{\lambda} \left[ \prod_{k\geq 1} \left( \frac{1-q^{kx}}{k\delta} \right)^{m_k(\sigma_1)}  e^{\cO\left(k m_k(\sigma_1) \delta \right)} \right] \\
        &= \lambda_1! \left( \frac{1-q^x}{\delta} \right)^{\lambda_1} e^{\cO(\lambda_1\delta)} \E_{\lambda} \left[ \prod_{k\geq 1} (k \mF_k(q^x))^{m_k(\sigma_1)} \right],
    \end{align*}
    where in the last line, we utilized $\sum_k k m_k(\sigma_1) = \lambda_1$ and $\frac{1-q^{kx}}{k\delta} = \big( \frac{1-q^x}{\delta} \big)^k k \mF_k(q^x) $. Then, as before, $\lambda_1 \delta \leq n\delta \lesssim n^{-1/2}$ shows that $e^{\cO(\lambda_1\delta)}$ produces the desired error estimate $1 + \cO\big(n^{-1/2}\big)$. Therefore, multiplying this with the estimate for $\prod_{k\geq 2} \mK_{a,i}$ yields \eqref{eqn:est_kai_detailed} as required.
\end{proof}

This lemma already hints that the contribution of fixed points is the only factor that affects the asymptotic distribution of the normalized pair $W_{\lambda}$. This point is more clearly seen from the following result, which is also last piece of the main claim of this section.

\begin{proposition}
    \label{prop:est_l_large}
    We have
    \begin{align}
        \label{eqn:est_l_large}
        \mL_{\mathrm{large}}
        &= \left( 1 + \cO\big(n^{-1/2}\big) \right) 
        \frac{ \int_{0}^{e^{-\epsilon}} u^{\frac{sn}{r}-1}(1-u)^{n} \mG(u) \, \mathrm{d}u }{ \int_{0}^{1} u^{\frac{sn}{r}-1}(1-u)^{n} \, \mathrm{d}u }.
    \end{align}
\end{proposition}

\begin{proof}[Proof of Proposition \ref{prop:est_l_large}]
    We apply Lemma \ref{lemma:genfn_num_cycles} to expand
    \begin{align*}
        \lambda_1! \E_{\lambda} \left[ \prod_{k\geq 1} (k \mF_k(q^x))^{m_k(\sigma_1)} \right]
        = \lambda_1! \sum_{\mu \vdash \lambda_1} \prod_{k\geq 1} \frac{\mF_k(q^x)^{\mu_k}}{\mu_k!}
        = t^{-\frac{\alpha_1^2}{2}n}q^{-\frac{\alpha_1^2}{4}n^2} \mG(q^x).
    \end{align*}
    Then, by Lemma \ref{lemma:est_kai_detailed}, we may write
    \begin{align*}
        t^{\frac{\alpha_1^2}{2}n}q^{\frac{\alpha_1^2}{4}n^2} \sum_{a \geq \epsilon/\delta } t^a \left( \prod_{k=1}^{n} \mK_{a,i} \right)
        &= \frac{1 + \cO\big(n^{-1/2}\big)}{\delta^n} \int_{\epsilon/\delta}^{\infty} t^x \left( 1-q^x  \right)^{n} \mG(q^x) \, \mathrm{d}x \\
        &= \frac{1 + \cO\big(n^{-1/2}\big)}{\delta^{n+1}} \int_{0}^{e^{-\epsilon}} u^{\frac{sn}{r}-1} \left( 1-u \right)^{n} \mG(u) \, \mathrm{d}u,
    \end{align*}
    where the last line is obtained by applying the substitution $u = q^x = e^{-x\delta}$. Plugging this to the definition of $\mL_{\mathrm{large}}$ proves \eqref{eqn:est_l_large}.
\end{proof}

This result tells us that $\mG$ accounts for the perturbation caused by the presence of fixed points. Before proceeding to the general case, we rejoice this result by looking into the case where $\cC_{\lambda}$ is free of fixed points, i.e. $\lambda_1 = 0$. In such case, we identically have $\mG(u) = 1$, and hence quickly establish the following partial result.

\begin{corollary}
    \label{thm:unifest_partial}
    For each $s > 0$ and $r > 0$, there exists a constant $C$, depending only on $s$ and $r$, such that
    \begin{align*}
        \left\lvert M_{W_{\lambda}}(-s, -r) - e^{\frac{1}{2}\Sigma_0(s, r)} \right\rvert \leq C n^{-1/2}
    \end{align*}
    holds for any $n \geq 1$ and for any fixed-point-free conjugacy class $\cC_{\lambda}$ of $S_n$. Consequently, along any sequence ofe~$\cC_{\lambda}$'s such that $n \to \infty$ and $\alpha_1 = 0$, $W_{\lambda}$ converges in distribution to a bivariabe normal distribution of zero mean and the covariance matrix $\Sigma_{0}$.
\end{corollary}

\begin{proof}[Proof of Corollary \ref{thm:unifest_partial}]
    Combining all the estimates and setting $\lambda_1 = 0$, we have
    \begin{align*}
        M(-s, -r)
        &= \cO\big(n^{-1/2}\big) + e^{\frac{1}{2}\Sigma_{0}(s, r)}  \left(1 - \frac{ \int_{e^{-\epsilon}}^{1} u^{\frac{sn}{r}-1}(1-u)^{n} \, \mathrm{d}u }{ \int_{0}^{1} u^{\frac{sn}{r}-1}(1-u)^{n} \, \mathrm{d}u } \right).
    \end{align*}
    It is easy to check that the ratio of two integrals above vanishes exponentially fast. Therefore, the first claim follows, and the second claim is a direct application of Proposition \ref{thm:continuity}.
\end{proof}

\subsection{Investigating the contribution of fixed points}

The aim of this subsection is to analyze the effect of the perturbation term $\mG$ in Proposition \ref{prop:est_l_large} and finalize the proof of the main theorem. To do so, we prove the following estimate on~$\mL_{\mathrm{large}}$.

\begin{proposition}
    \label{prop:est_l_large_2}
    We have
    \begin{align}
        \label{eqn:est_l_large_2}
        \mL_{\mathrm{large}}
        = \left( 1 + \cO\big(n^{-1/6}\big) \right) e^{\frac{1}{2}\left(\Sigma_{\alpha_1}(s, r) - \Sigma_0(s, r) \right)}.
    \end{align}
\end{proposition}

Before proceeding, we check that this is the last ingredient towards establishing Theorem \ref{thm:unifest} and consequently Theorems \ref{thm:main_01} and \ref{thm:main_02}.

\begin{proof}[Proof of Theorems \ref{thm:unifest}, \ref{thm:main_01}, and \ref{thm:main_02}]
    Combining Proposition \ref{prop:mgf_expand_2} and \ref{prop:est_l_large_2}, we have
    \begin{align*}
        M_{W_{\lambda}}(-s, -r) = \left( 1 + \cO\big(n^{-1/6}\big) \right) e^{\frac{1}{2}\Sigma_{\alpha_1}(s, r)} + \cO(\mL_{\mathrm{small}}).
    \end{align*}
    By noting that $\mL_{\mathrm{small}}$ decays at least exponentially fast, it can be absorbed into the error term $\cO(n^{-1/6})$, and hence, Theorem \ref{thm:unifest} follows.
    
    Since Theorem \ref{thm:main_01} is a special case of Theorem \ref{thm:main_02}, it is enough to prove the latter. Fix $s, r > 0$ and let $C = C(s, r) > 0$ be as in Theorem \ref{thm:unifest}. Let $A_n$ and $\alpha_{1,n}$ satisfy the hypotheses of Theorem \ref{thm:main_02}. Then, for each $n \in \bN$, there exists a family~$\Lambda_n$ of integer partitions of $n$ such that $A_n = \bigcup_{\lambda \in \Lambda_n} \cC_{\lambda}$ and that $\alpha_1(\lambda) = \alpha_{1,n}$ for all $\lambda \in \Lambda_n$. Then,
    \begin{align*}
        \left| M_{W_n}(-s, -r) - e^{\frac{1}{2}\Sigma_{\alpha_{1,n}}(s, r)} \right|
        &\leq \sum_{\lambda \in \Lambda_n} \frac{|\cC_{\lambda}|}{|A_n|} \left|  M_{W_{\lambda}}(-s, -r) - e^{\frac{1}{2}\Sigma_{\alpha_{1}(\lambda)}(s, r)} \right| \\
        &\leq C n^{-1/6}.
    \end{align*}
    Since $\alpha_{1,n} \to \alpha$ by the hypothesis, $M_{W_n}(-s, -r) \to e^{\frac{1}{2}\Sigma_{\alpha}(s, r)}$, and therefore, the conclusion follows from the modified Curtiss' theorem.
\end{proof}

We return to the problem of showing \eqref{eqn:est_l_large_2}. The first step is to produce a manageable formula for $\mG$ modulo relative error of magnitude $n^{-1/6}$. Although we suspect that $n^{-1/6}$ is not optimal, we do not attempt to make any improvements.

\begin{lemma}
    \label{lemma:est_G}
    There exists $N \geq 1$, depending only on $s$ and $r$, such that
    \begin{align}
        \label{eqn:est_G}
        \mG(u)
        = \left( 1 + \cO\big(n^{-1/6}) \right) \exp\left\{ \lambda_1^2 \left( \mF_2(u) -  \mF_2\left( \frac{s}{s+r} \right) \right) + \lambda_1^3 \left( \mF_3(u) - 2 \mF_2(u)^2 \right) \right\}
    \end{align}
    for all $n \geq N$, $\lambda \vdash n$, and $u \in [0, e^{-\epsilon}]$.
\end{lemma}

\begin{proof}[Proof of Lemma \ref{lemma:est_G}]
    Let $\{X_k\}_{k\geq 2}$ be a family of independent random variables for which $X_k$ has the Poisson distribution of the rate $\lambda_1^k \mF_k(u)$. Also, let $Y_k = k X_k$ and $Y = \sum_{k\geq 2} Y_k$. Then, we can rephrase $\mG(u)$ as
    \begin{align*}
        \mG(u)
        &= t^{\frac{\alpha_1^2}{2}n}q^{\frac{\alpha_1^2}{4}n^2} \lambda_1! \sum_{\mu_2, \mu_3, \cdots \geq 0} \frac{1}{(\lambda_1 - \sum_{k=2} k\mu_k)!} \prod_{k\geq 2} \frac{\mF_k(u)^{\mu_k}}{\mu_k!} \\
        &= t^{\frac{\alpha_1^2}{2}n}q^{\frac{\alpha_1^2}{4}n^2} e^{\sum_{k \geq 2} \E [X_k]} \E \left[ \frac{\lambda_1!}{(\lambda_1 - Y)! \lambda_1^Y} \right],
    \end{align*}
    where we adopt the convention that $\frac{1}{k!} = 0$ for negative integers $k$. Also, we will frequently use the trivial bound $\lambda_1 \leq n = (\delta/r)^{-2/3}$.
    
    We begin by providing an estimate on $\E [X_k] = \Var(X_k) = \lambda_1^k \mF_k(u)$. Assume $u \in [0, e^{-\epsilon}]$. Also, choose $N$ so that $\frac{r}{N^{1/2}(1-e^{-\epsilon})} \leq \frac{1}{2}$ and assume $n \geq N$. Then, we claim that
    \begin{align}
        \label{eqn:est_03}
        \sum_{k \geq j} k^p \E [X_k]
        \lesssim_{j,p} \E [X_j]
        \lesssim_{j} \lambda_1^j \delta^{j-1}
        \lesssim_{j} n^{(3-j)/2}
    \end{align}
    holds for each fixed $j \geq 2$ and $p \in \bR$, with implicit bounds depending only on $j$, $p$, $s$ and $r$. Indeed, for each $k \geq 2$,
    \begin{align*}
        \E [X_k]
        = \frac{\lambda_1^k \delta^{k-1}}{k^2} \cdot \frac{1-u^k}{(1-u)^k}
        \leq \frac{n^k \delta^{k-1}}{(1 - e^{-\epsilon})^k}
        \leq \frac{r^{k-1}}{(1 - e^{-\epsilon})^k} n^{(3-k)/2},
    \end{align*}
    and for each $k \geq j \geq 2$,
    \begin{align*}
        \frac{\E [X_k]}{\E [X_j]}
        &= \frac{j^2(1-u^k)}{k^2(1-u^j)} \left( \frac{\lambda_1 \delta}{1 - u} \right)^{k-j}
        \leq \frac{j^2}{1-e^{-j\epsilon}} \left( \frac{r}{n^{1/2}(1 - e^{-\epsilon})} \right)^{k-j}.
    \end{align*}
    Then, \eqref{eqn:est_03} follows easily from both inequalities.
    
    Now, from $ t^{\alpha_1^2 n/2}q^{\alpha_1^2 n^2/4} = \exp \big\{ - \lambda_1^2 \mF_2\big( \frac{s}{s+r} \big) \big\}$ and \eqref{eqn:est_03} with $j = 4$, we get
    \begin{align*}
        &t^{\frac{\alpha_1^2}{2}n}q^{\frac{\alpha_1^2}{4}n^2} e^{\sum_{k \geq 2} \E [X_k]} \\
        &\hspace{1.5em} = e^{\sum_{k \geq 4} \E [X_k]} \exp\left\{ \E [X_2] - \lambda_1^2 \mF_2\left( \frac{s}{s+r} \right) + \E [X_3] \right\} \\
        &\hspace{1.5em} = \big( 1 + \cO\big(n^{-1/2}\big) \big) \exp\left\{ \lambda_1^2 \left( \mF_2(u) -  \mF_2\left( \frac{s}{s+r} \right) \right) + \lambda_1^3 \mF_3(u) \right\}.
    \end{align*}
    As a result, it suffices to show that
    \begin{align}
        \label{eqn:est_ex}
        \E \left[ \frac{\lambda_1!}{(\lambda_1 - Y)! \lambda_1^Y} \right]
        = \big( 1 + \cO\big(n^{-1/6}\big)\big) e^{-2\lambda_1^3 \mF_2(u)^2}.
    \end{align}
    For the rest of the proof, we will establish \eqref{eqn:est_ex}. To this end, we will divide the proof into two cases according to whether $\lambda_1 < \delta^{-1/4}$ or $\lambda_1 \geq \delta^{-1/4}$.
    
    \noindent \\ \textbf{Case 1: $\lambda_1 < \delta^{-1/4}$.} In such case, $\lambda_1^3 \mF_2(u)^2 \lesssim \lambda_1^3 \delta^2 < \delta^{5/4} \lesssim n^{-1/6}$, and so,
    \begin{align*}
        e^{-2\lambda_1^3 \mF_2(u)^2} = 1 + \cO\big(n^{-1/6}\big).
    \end{align*}
    For the left-hand side of \eqref{eqn:est_ex}, $\sum_{k \geq 2} \E [X_k] \lesssim \lambda_1^2 \delta < \delta^{1/2} \lesssim n^{-1/6}$, and so,
    \begin{align*}
        1
        \geq \E \left[ \frac{\lambda_1!}{(\lambda_1 - Y)! \lambda_1^Y} \right]
        \geq \Prob [ Y = 0 ]
        = e^{-\sum_{k \geq 2} \E [X_k]}
        = 1 - \cO\big(n^{-1/6}\big).
    \end{align*}
    Consequently, both sides of \eqref{eqn:est_ex} are $1+\cO(n^{-1/6})$, and hence, the claim is true.
        
    \noindent \\ \textbf{Case 2: $\lambda_1 \geq \delta^{-1/4}$.} In such case, both the assumption and the trivial bound gives us $\lambda_1^{-1/2} \leq \delta^{1/8} \lesssim n^{-1/6}$ and $(\lambda_1 \delta)^{1/3} \lesssim n^{-1/6}$. Also, by \eqref{eqn:est_03},
    \begin{align*}
        \Var(Y) = \sum_{k\geq 2} k^2 \E [X_k] \lesssim \lambda_1^2 \delta
        \quad \text{and} \quad
        \sum_{k \geq 3} \E [Y_k] = \sum_{k \geq 3} k \E [X_k] \lesssim 1.
    \end{align*}
    Let $\cA$ the event that $|Y - \E [Y]| \leq \lambda_1^{5/6}\delta^{1/3}$ holds. Then, by Chebyshev's inequality,
    \begin{align}
        \label{eqn:est_04}
        \Prob[ \cA^c ]
        \leq \frac{\Var(Y)}{\big( \lambda_1^{5/6} \delta^{1/3} \big)^2}
        \lesssim (\lambda_1 \delta)^{1/3}
        \lesssim n^{-1/6},
    \end{align}
    Also, given $\cA$, dividing both sides of $Y = \E [Y] + \cO\big(\lambda_1^{5/6}\delta^{1/3}\big)$ by $\lambda_1^{1/2}$ gives
    \begin{align}
        \label{eqn:est_05}
        \begin{split}
            \frac{Y}{\lambda_1^{1/2}}
            &= \frac{2 \E [X_2]}{\lambda_1^{1/2}} + \lambda_1^{-1/2} \sum_{k \geq 3} \E [Y_k] + \cO\big( (\lambda_1 \delta)^{1/3} \big) \\
            &= \frac{2 \E [X_2]}{\lambda_1^{1/2}} + \cO\big( n^{-1/6} \big).
        \end{split}
    \end{align}
    Plugging $\E [X_2] \lesssim \lambda_1^2 \delta \lesssim \lambda_1^{1/2}$ into \eqref{eqn:est_05} yields $Y \lesssim \lambda_1^{1/2}$, and so,
    \begin{align*}
        \frac{\lambda_1!}{(\lambda_1 - Y)! \lambda_1^Y}
         = \prod_{j=0}^{Y-1} \left( 1 - \frac{j}{\lambda_1} \right)
        &= \exp\left\{ - \sum_{j=0}^{Y-1} \left( \frac{j}{\lambda_1} + \cO\left( \frac{j^2}{\lambda_1^2} \right) \right) \right\}  \\
        &= \exp\left\{ - \frac{Y^2}{2\lambda_1} + \cO\big( \lambda_1^{-1/2} \big) \right\} \\
        &= \exp\left\{ - \frac{2(\E [X_2])^2}{\lambda_1} + \cO \big( n^{-1/6} \big) \right\}.
    \end{align*}
    Combining altogether and noting that $\frac{(\E [X_2])^2}{\lambda_1} = \lambda_1^3 \mF_2(u)^2 \lesssim \lambda_1^3 \delta^2 \lesssim 1$,
    \begin{align*}
        \E \left[ \frac{\lambda_1!}{(\lambda_1 - Y)! \lambda_1^Y} \right]
        &= \E \left[ \frac{\lambda_1!}{(\lambda_1 - Y)! \lambda_1^Y} \, \middle| \, \cA \right]\Prob[\cA] + \cO (1) \Prob[\cA^c] \\
        &= e^{-2\lambda_1^3 \mF_2(u)^2 + \cO(n^{-1/6})} \left(1 + \cO\big(n^{-1/6}\big)\right) + \cO\big(n^{-1/6}\big) \\
        &= \left(1 + \cO\big(n^{-1/6}\big)\right) e^{-2\lambda_1^3 \mF_2(u)^2},
    \end{align*}
    which proves \eqref{eqn:est_G} as desired.
\end{proof}

The above lemma takes cares of the term $\mG$. Then, $\mL_{\mathrm{large}}$ reduces to the ratio of two explicit integrals modulo some uniform relative error, hence all that is left is to give an explicit control over those integrals. These integrals can be analyzed in the fashion of Laplace's method, with the term $\mG$ behaving as an exponential tilting of the gaussian density function. (We will see what exactly this mean in the proof of Proposition \ref{prop:est_l_large_2}.) The following lemma will be useful for this purpose.

\begin{lemma}
    \label{lemma:dominating}
    For each $n \geq 1$, define $f_n : \bR \to \bR$ by
    \begin{align*}
        f_n(z) = \left( 1 + \frac{z}{\sqrt{n}} \right)_+^n e^{-\sqrt{n}z}
    \end{align*}
    where $x_+ := \max\{x, 0\}$ denotes the positive part of $x$. Then, for any $z \in \bR$ and for any $n \geq m \geq 1$, we have that
    \begin{align*}
        f_n(z)
        \leq \begin{cases}
            e^{-\frac{1}{2}z^2}, & z < 0 \\
            f_m(z), & z \geq 0
        \end{cases}.
    \end{align*}
\end{lemma}

\begin{proof}[Proof of Lemma \ref{lemma:dominating}]
    We may regard $n$ as a variable taking values in $(0, \infty)$ and $f_n(z)$ as function of $n$. Then, it is straightforward to verify that
    \begin{align*}
        \lim_{n\to\infty} \log f_n(z) = -\frac{z^2}{2}
        \qquad \text{and} \qquad
        \lim_{n\to\infty} \frac{\partial \log f_n(z)}{\partial n} = 0
    \end{align*}
    for any $z \in \bR$. Moreover, if $n > 0$ satisfies $z+\sqrt{n} > 0$, then
    \begin{align*}
        \frac{\partial^2 \log f_n(z)}{\partial n^2} = \frac{z^3}{4n^{3/2}(z+\sqrt{n})^2},
    \end{align*}
    So $n \mapsto f_n(z)$ is decreasing when $z \geq 0$ and increasing when $z \leq 0$, which in turn implies the desired bound.
\end{proof}

This lemma allows us to estimate the term $\left(1+\frac{z}{\sqrt{n}}\right)_+^{n}$, and in addition, provides explicit control over the relative error of this approximation. Finally, we are ready to complete the proof.

\begin{proof}[Proof of Proposition \ref{prop:est_l_large_2}]
    Write $\theta = \frac{s}{s+r}$. Define $z = z_n(u)$ by $u = \theta + (1-\theta) \sqrt{\frac{\theta}{n}} z$. Plugging the definition of $\mF_2$ and $\mF_3$ into the exponent of \eqref{eqn:est_G} and performing some algebra, we obtain
    \begin{align*}
        &\lambda_1^2 \left( \mF_2(u) -  \mF_2(\theta) \right) - \lambda_1^3 \left( 2 \mF_2(u)^2 - \mF_3(u) \right) \\
        &\hspace{5em}= \frac{\alpha_1^2 r}{2}\frac{\sqrt{\theta}}{1-u}z_n(u) - \frac{\alpha_1^3 r^2}{72} \frac{1+10u+u^2}{(1-u)^2}.
    \end{align*}
    Since $u \in [0, e^{-\epsilon}]$ is uniformly away from $1$, this shows that there exists $C_1 > 0$, depending only on $s$ and $r$, such that $\mG(u) \leq e^{C_1(\lvert z_n(u) \rvert+1)}$ holds for all $u \in [0, e^{-\epsilon}]$. Also,
    \begin{align*}
        u^{\frac{sn}{r}-1}(1-u)^{n}
        &= \theta^{\frac{sn}{r}-1}(1-\theta)^n \left( 1 + \frac{r}{s}\sqrt{\frac{\theta}{n}} z_n(u) \right)^{\frac{sn}{r}-1} \left( 1 - \sqrt{\frac{\theta}{n}} z_n(u) \right)^n \\
        &= \theta^{\frac{sn}{r}-1}(1-\theta)^n e^{-u/\theta} g_n(z),
    \end{align*}
    where $g_n(z) = f_n\left( \frac{r}{s}\sqrt{\theta}z_n(u) \right)^{\frac{s}{r}-\frac{1}{n}} f_n\left( - \sqrt{\theta}z_n(u) \right)$. To estimate this, note that Lemma \ref{lemma:dominating} shows that $f_n$ can be dominated by exponentially decaying function of arbitrarily fast rate. In particular, there exist $N \geq 1$ and $C_2 > 0$, depending only on $s$ and $r$, such that $g_n(z) \leq C_2 e^{-2C_1\lvert z \rvert}$ for all $n \geq N$ and for all $z \in \bR$.
    
    Now, we want to estimate both the denominator and the numerator of \eqref{eqn:est_l_large}. By the aforementioned substitution, we have
    \begin{align*}
        \frac{ \int_{0}^{e^{-\epsilon}} u^{\frac{sn}{r}-1}(1-u)^{n} \mG(u) \, \mathrm{d}u }{ \int_{0}^{1} u^{\frac{sn}{r}-1}(1-u)^{n} \, \mathrm{d}u }
        = \frac{ \int_{z_n(0)}^{z_n(e^{-\epsilon})} e^{1-u/\theta} g_n(z) \mG(u) \, \mathrm{d}z }{ \int_{z_n(0)}^{z_n(1)} e^{1-u/\theta} g_n(z) \, \mathrm{d}z }.
    \end{align*}
    Thus, it suffices to estimate the right-hand side. We proceed by splitting the integrals according to whether $|z| > \frac{\log n}{C_1}$ or $|z| \leq \frac{\log n}{C_1}$. In the first case,
    \begin{align*}
        \int\limits_{|z| > \frac{\log n}{C_1} } e^{1-u/\theta} g_n(z) \mG(u) \, \mathrm{d}z
        \lesssim \int\limits_{|z| > \frac{\log n}{C_1}} e^{-C_1\lvert z \rvert} \, \mathrm{d}z
        \lesssim n^{-1},
    \end{align*}
    and similarly,
    \begin{align*}
        \int\limits_{|z| > \frac{\log n}{C_1} } e^{1-u/\theta} g_n(z) \, \mathrm{d}z
        \lesssim \int\limits_{|z| > \frac{\log n}{C_1}} e^{-2C_1\lvert z \rvert} \, \mathrm{d}z
        \lesssim n^{-2}.
    \end{align*}
    So, we focus on the region $|z| \leq \frac{\log n}{C_1}$. In such case, $u = \theta + \cO\left(\frac{\log n}{n^{1/2}} \right)$, and so,
    \begin{align*}
        e^{1-u/\theta} g_n(z)
        &= \exp\left\{ - \frac{r \theta z^2}{2s} - \frac{\theta z^2}{2} + \cO\left( \frac{(\log n)^3}{n^{1/2}}\right) \right\} \\
        &= \exp\left\{ - \frac{z^2}{2} + \cO\left( \frac{(\log n)^3}{n^{1/2}}\right) \right\},
    \end{align*}
    and 
    \begin{align*}
        \mG(u)
        &= \left( 1 + \cO\big(n^{-1/6}\big)\right) \exp\left\{ \frac{\alpha_1^2 r}{2}\frac{\sqrt{\theta}}{1-\theta}z - \frac{\alpha_1^3 r^2}{72} \frac{1+10\theta+\theta^2}{(1-\theta)^2} + \cO\left( \frac{(\log n)^2}{n^{1/2}}\right) \right\} \\
        &= \left( 1 + \cO\big(n^{-1/6}\big)\right) \exp\left\{ \frac{\alpha_1^2}{2}\sqrt{s(r+s)}z - \frac{\alpha_1^3}{72} \left(r^2 + 12rs + 12s^2\right) \right\}.
    \end{align*}
    Finally, we have the following gaussian integral
    \begin{align*}
        \int_{-\infty}^{\infty} e^{-\frac{z^2}{2}+\frac{\alpha_1^2}{2}\sqrt{s(r+s)}z - \frac{\alpha_1^3}{72} \left(r^2 + 12rs + 12s^2\right)} \, \mathrm{d}z
        = \sqrt{2\pi} e^{\frac{1}{2}\left(\Sigma_{\alpha_1}(s, r) - \Sigma_0(s, r) \right)}.
    \end{align*}
    Therefore, combining altogether,
    \begin{align*}
        \frac{ \int_{0}^{e^{-\epsilon}} u^{\frac{sn}{r}-1}(1-u)^{n} \mG(u) \, \mathrm{d}u }{ \int_{0}^{1} u^{\frac{sn}{r}-1}(1-u)^{n} \, \mathrm{d}u }
        &= \frac{ \int_{|z| \leq (\log n)/C_1} e^{1-u/\theta} g_n(z) \mG(u) \, \mathrm{d}z + \cO(n^{-1}) }{ \int_{|z| \leq (\log n)/C_1} e^{1-u/\theta} g_n(z) \, \mathrm{d}z + \cO(n^{-2}) } \\
        &= \frac{ \sqrt{2\pi} e^{\frac{1}{2}\left(\Sigma_{\alpha_1}(s, r) - \Sigma_0(s, r) \right)} + \cO\big( n^{-1/6} \big) }{ \sqrt{2\pi} + \cO\Big( \frac{(\log n)^3}{n^{1/2}}\Big) } \\
        &= \left( 1 + \cO\big(n^{-1/6}\big) \right) e^{\frac{1}{2}\left(\Sigma_{\alpha_1}(s, r) - \Sigma_0(s, r) \right)}.
    \end{align*}
    Plugging this into \eqref{eqn:est_l_large} proves the desired claim.
\end{proof}


\section{Discussion of possible extensions}

\begin{enumerate}[label={\arabic*.},leftmargin=*,itemsep=0.5em]
    \item One of the main tools of this proof is the modified Curtiss' theorem, which is a significant result in its own right. The authors believe that this can be applied to other problems where explicit formulas of the generating function are known.
    
    \item As commented in Subsection 4.3, the bound $C n^{-1/6}$ in the result of Theorem \ref{thm:unifest} is rather artificial and may possibly be improved. Considering that $n^{-1/2}$ is the prevalent scale in most relative error terms, it is natural to expect that the correct bound would be $n^{-1/2}$ with possible logarithmic corrections.
    
    \item Theorem \ref{thm:main_02} is formulated for conjugation-invariant subsets $A_n$ of $S_n$ with the additional hypothesis that all elements of $A_n$ have the same number of fixed points. In terms of the distribution of the density $m_1(\pi)/n$ of fixed points of $\pi \in A_n$, this hypothesis is equivalent to assuming that $m_1(\pi)/n$ is concentrated at the single value $\alpha_{1,n}$. In view of the uniform estimate in Theorem \ref{thm:unifest}, however, this is not a serious restriction, and indeed there is room for further generalizations. One immediate generalization of Theorem \ref{thm:main_02} is obtained from replacing the hypothesis on $A_n$ by the following one: 
    \begin{center}
        \begin{minipage}{0.85\textwidth}
            $A_n$ is a conjugation-invariant subset of $S_n$ such that $m_1(\pi) \leq \alpha_{1,n}n$ for all $\pi \in A_n$ and that $|\{ \pi \in A_n : m_1(\pi) = \alpha_{1,n}n \}|/|A_n| \to 1$ as $n\to\infty$.
        \end{minipage}
    \end{center}
    
    \item The authors would like to a combinatorial explanation for the $(1-\alpha^2)$ factor in the asymptotic mean.
\end{enumerate}

\bibliographystyle{amsplain}

\end{document}